\documentclass[11pt,a4paper]{amsart}
\usepackage{a4}
\usepackage[latin1]{inputenc}
\usepackage[T1]{fontenc}
\usepackage{amsmath}
\usepackage{amsfonts,amssymb,amsthm}
\usepackage[color=green!30]{todonotes}
\usepackage{xcolor}
\usepackage{enumitem}
\usepackage[colorlinks,
    linkcolor={red!50!black},
    citecolor={blue!50!black},
    urlcolor={blue!80!black}]{hyperref}
\usepackage[nameinlink,capitalise]{cleveref}
\usepackage[xy]{qsymbols}
\usepackage{tikz-cd}
\usepackage{mathtools}
\usepackage{array}
\usepackage{arydshln}
\usepackage{bigdelim,multirow}

\newcommand{\GL}{\operatorname{GL}}
\newcommand{\SL}{\operatorname{SL}}
\newcommand{\SU}{\operatorname{SU}}
\newcommand{\QQ}{\mathbb{Q}}
\newcommand{\CC}{\mathbb{C}}
\newcommand{\RR}{\mathbb{R}}
\newcommand{\ZZ}{\mathbb{Z}}

\newcommand{\cM}{M_K}
\newcommand{\bma}{\begin{pmatrix}}
\newcommand{\ema}{\end{pmatrix}}
\newcommand{\bsm}{\left(\begin{smallmatrix}}
\newcommand{\esm}{\end{smallmatrix}\right)}
\newcommand{\Tr}{\operatorname{Tr}}
\newcommand{\Ad}{\operatorname{Ad}}
\newcommand{\Adr}{\Ad\circ\rho}

\newcommand{\tor}{\operatorname{tor}}
\newcommand{\ord}{\operatorname{ord}}

\newtheorem{theorem}{Theorem}[section]

\newtheorem{lemma}[theorem]{Lemma}
\newtheorem{proposition}[theorem]{Proposition}

\newtheorem*{claim}{Claim}

\theoremstyle{definition}
\newtheorem{definition}[theorem]{Definition}

\newtheorem{example}[theorem]{Example}

\newtheorem{remark}[theorem]{Remark}


\begin{document}

\title{A Slope invariant and the A-polynomial of knots}
\begin{abstract}
	The A-polynomial is a knot invariant related to the space of $\SL_2(\CC)$ representations of the knot group. 	
	In this paper our interests lies in the \emph{logarithmic Gauss map} of the $A$-polynomial. We develop a homological point of view on this function by extending the constructions of Degtyarev, the second author and Lecuona to the setting of non-abelian representations. It defines a rational function on the character variety, which unifies various known invariants such as the \emph{change of curves} in the Reidemeister function, the \emph{modulus} of boundary-parabolic representations, the \emph{boundary slope} of some incompressible surfaces embedded in the exterior of the knot $K$ or equivalently the slopes of the sides of the Newton polygon of the A-polynomial $A_K$. We also present a method to compute this invariant in terms of Alexander matrices and Fox calculus.
\end{abstract}
\author{Leo Benard}
\address{Mathematisches Institut, Georg-August Unversit\"at, G\"ottingen, Deutschland}
\email{leo.benard@mathematik.uni-goettingen.de}
\author{Vincent Florens}
\address{Universit\'e de Pau, Pau, France}
\email{vincent.florens@univ-pau.fr}
\author{Adrien Rodau}
\address{Universit\'e de Pau, France}
\email{arodau@univ-pau.fr}
\date{}

\maketitle

\section{Introduction}
The set of all representations of a knot group in $\SL_2(\CC)$ carries naturally the structure of an algebraic set.  This holds also for the characters of these representations, whose set is called the $\SL_2(\CC)$-character variety of the knot.
Given a peripheral structure of the knot, the character variety
is  a plane curve in $\CC^* \times \CC^*$, whose coordinates $M$ and $L$ correspond to the eigenvalues of the meridian $m$ and the prefered longitude $\ell$. The polynomial $A_K(L,M)$ defining this curve is an invariant of the knot, called the A-polynomial.
 This invariant contains many interesting informations on the knot; in particular, Boyer and Zhang \cite{Boyer_Zhang} and Dunfield and Garoufalidis \cite{Dunfield_Garouf} showed that $A_K=L-1$ if and only if $K$ is trivial.

In this paper, our motivations come, among others, from the following result of Boden:
\begin{theorem}[\cite{Boden}]
	\label{theo:detect}
	If the $M$-degree $\deg_M A_K(L,M)$  of the $A$-polynomial  is zero, then $K$ is the trivial knot.
\end{theorem}
This result motivates the systematic study of the \emph{logarithmic Gauss map} of the A-polynomial
\begin{equation}
	\label{eq:slope}
	 \frac M L \cdot \frac{\partial_M A_K(L,M)}{\partial_L A_K(L,M)},
\end{equation}
 where $\partial_M$ and $\partial_L$ denote the partial derivatives.
By \cref{theo:detect}, this rational function vanishes identically on $\{ A_K =0 \}$ if and only if $K$ is trivial.

The logarithmic Gauss map has introduced in \cite{GKZ94} by Guelfand, Kapranov and Zelevinsky in order to study some determinantial varieties. Then it has been used for instance by Mikhalkin in \cite{Mikh} for studying the topology of arrangements of real plane curves. In \cite{GuiMar}, Marché and Guilloux showed it is related with the volume function of the A-polynomial of knots, or more generally of exact polynomials.

Our proposal is to develop a homological point of view on this function, by extending the constructions of Degtyarev, the second author and Lecuona \cite{DFL1,DFL2} to the setting of non-abelian representations. Let $K$ be an oriented knot in the 3-sphere $S^3$ with exterior $\cM$. Denote by $R(\cM)$ and $X(\cM)$ the $\SL_2(\CC)$-representation and character varieties of the knot~$K$. We consider representations $\rho \colon \pi_1(\cM) \to \SL_2(\CC)$ composed with the adjoint action of $\SL_2(\CC)$ on the Lie algebra $\Ad\colon \SL_2(\CC) \to \operatorname{Aut}(\mathfrak{sl}_2(\CC))$, and show that there is a non-empty Zariski open subset of $X(\cM)$ such that for all $\rho$ in this subset
\begin{itemize}[label=-]
	\item there is an element $v_\rho \in \mathfrak{sl}_2(\CC)$ such that $(v_\rho \otimes \ell,v_\rho \otimes m)$ is a basis of the homology group $H_1(\partial \cM, \Adr) \simeq \CC^2$ with coefficients twisted by $\Adr$, and
	\item the kernel of the homomorphism induced by the inclusion:
	\[\mathcal Z(K,\Adr) = \ker \left( H_1(\partial \cM, \Adr) \stackrel{i_*} \longrightarrow H_1(\cM, \Adr)\right)\]
	is generated by a single vector of the form $a\, v_\rho \otimes \ell + b \, v_\rho \otimes m$ for some $[b:a] \in \CC \mathbb P^1$.
\end{itemize}
The representations which verify these conditions are called \emph{admissible}. We define the slope of $K$ at the admissible representation $\rho$  by
\[s_K(\rho) = -\frac{b}{a} \in \CC \mathbb P^1.\]
We prove that representations which restrict to \emph{non-parabolic} representations of the boundary $\partial \cM$ of $\cM$ are admissible, see \cref{lem:dim1}.
If $\rho$ is a boundary-parabolic representation, we define the slope $s_K(\rho)$ as the modulus of the euclidean structure induced by the restricted representation
on $\pi_1(\partial \cM)$, see \cref{subsec:reg}. It turns out that these two different definitions fit well and that the following holds.

\begin{proposition}
	\label{prop:rational}
	The slope depends only on the conjugacy classes of the representations and induces a rational function
  \[s_K : 	X \subset X(\cM) \longrightarrow \CC \mathbb P^1\]
	on each {irreducible} component $X$ of the character variety.
\end{proposition}

Note that if the representation is real or unitary, then $s_K$ takes values in $\RR \mathbb P^1$ (see \cref{prop:real}).
{For any knot, the function $s_K$ can be computed by Fox calculus, see \cref{subsec:Alex}. We illustrate the method in the case of the trefoil knot, and further compute the slope of the figure-eight knot.}

{The following theorem relates $s_K$ to the original motivation; a precise statement is given in \cref{th:slopeApoly}}.
\begin{theorem}
	\label{prop:slope^2}
	The slope function $s_K$ equals minus the logarithmic Gauss map of the A-polynomial defined in \cref{eq:slope}.
\end{theorem}

We also relate $s_K$ to the change of curve factor for the Reidemeister torsion.
Let $\mathbb T_{\cM, \ell}(\rho)$ and $\mathbb T_{\cM, m}(\rho)$ be the Reidemeister torsions according to homology bases induced by the choices of the curves $\ell$ and $m$ in $\partial \cM$, see \cref{subsec:tors}.
\begin{proposition}
		\label{prop:collTors}
		The slope 
		coincides with the quotient of Reidemeister torsion:
		\[s_K(\rho) = \frac{\mathbb T_{\cM, \ell}(\rho)}{\mathbb T_{\cM, m}(\rho)}\]
		for all $\rho$ such that this formula is well-defined.
\end{proposition}

Porti had already observed (\cite[Corollary 4.9]{Porti97}) that the logarithmic Gauss map of the A-polynomial could be expressed as a ratio of torsions -up to a sign-, and that this ratio of torsions is equal to the modulus of $\rho$ when it is a boundary-parabolic representation (\cite[Proposition 4.7]{Porti97}). Our point of view permits to fix and compute the sign ambiguity. Moreover, our results \cref{prop:rational} and \cref{prop:slope^2} are more general, since they do not require the Reidemeister torsion to be well-defined, for instance they hold for high dimensional components of the character variety.

Finally, we consider \emph{ideal} points of the A-polynomial, those are points added at infinity in a compactification of the curve $\{A(L,M)=0\}$ in $\CC^2$. In \cite{CS83}, Culler and Shalen constructed incompressible surfaces in $\cM$ associated to such points.
 Those surfaces have a non-empty boundary, whose slope is determined by 
rational number $p/q$.
We prove the following theorem:
\begin{theorem}
	\label{theo:slope}
	Let $y \in \CC\mathbb P^2$ be an ideal point of 
	the curve $\{ A_K(L,M)=0 \}$.
	The
	value of 
	$s_K$ at $y$ equal \emph{minus} the slope of the Culler--Shalen incompressible surface associated to $y$.
\end{theorem}
This theorem sheds some light on the main theorem of \cite{CCGLS}, which states that the boundary slopes of the Culler--Shalen surfaces are boundary slopes of the Newton polygon of the A-polynomial. Indeed it is well-known that the logarithmic Gauss map converges at those ideal points to the value of the slope of the corresponding boundary of the Newton polygon.

To conclude this introduction, we mention that the slope invariant can be extended to orthogonal (real) representations of link groups. In this more general setting, the first twisted homology space $H_1(\partial \cM, \rho)$ can have an arbitrary dimension higher than $2$ and the kernel $\mathcal Z(K,\rho)$ might not be a line anymore. However, the space $H_1(\partial \cM, \rho)$ carries a natural symplectic structure given by the (twisted) intersection form on $\partial \cM$, and $\mathcal Z(K,\rho)$ is still a Lagrangian subspace. A construction of Arnold \cite{Arnold} related to the Maslov index allows to construct a \emph{generalized slope} for this context, lying in $S^1 \subset \CC^*$.  As it turns out, in the case of a representation $\rho\colon \pi_1(M_K) \to \SU(2)$, both theories coincide via the natural isomorphism $\mathbb R \mathbb P^1 \simeq S^1$.
We postpone rigorous definitions and further study of this invariant to an upcoming article.
\subsection*{Organisation of the paper}
In \cref{sec:char} we collect basic definitions on character varieties and A-polynomials. In \cref{sec:slope} we define the slope invariant and we prove \cref{prop:rational} and \cref{prop:collTors}. In \cref{subsec:slopeApol} we prove \cref{prop:slope^2}. 
Finally, in \cref{subsec:ideal} we prove \cref{theo:slope}.

\subsection*{Aknowledgements}
Léo Bénard is a member of the Research Training Group 2491 "Fourier Analysis and Spectral Theory", University of Göttingen. 
The authors thank Julien Marché, Stepan Yu.\,Orevkov and Joan Porti for related conversations.

\section{Representation varieties,  character varieties and A-polynomial.}
\label{sec:char}
This section is devoted to definitions and properties of representations spaces and character varieties (\cref{subsec:char}). We compute the character variety of the group $\ZZ^2$ in \cref{caractorus} and define the A-polynomial of knots in \cref{subsec:Apol}.
References for character varieties are \cite{Sha02, Sikora}, the A-polynomial was first defined in \cite{CCGLS}, see also \cite{CL98}.

\subsection{Representation and character varieties}
\label{subsec:char}
Let $\Gamma$ be a finitely generated group. The \emph{representation variety} is the affine algebraic set
\[R(\Gamma)=\text{Hom}(\Gamma, \SL_2(\mathbb{C})).\]
If $\Gamma$ is generated by $n$ elements, the representation variety is an algebraic subset of $\SL_2(\CC)^{n}$ given by polynomial relations corresponding to the relations of the group $\Gamma$. Two different presentations yield naturally isomorphic algebraic sets.

A representation $\rho:\Gamma \rightarrow \SL_2(\mathbb{C})$ is \emph{abelian} is $\rho(\Gamma)$ is an abelian subgroup of $\SL_2(\CC)$.  A representation $\rho$ is \emph{reducible} if there exists a proper subspace of $\mathbb{C}^2$ invariant under the action of $\rho(\Gamma)$.
Equivalently, $\rho(\Gamma)$ is conjugated to a subgroup of the group of upper-triangular matrices in $\SL_2(\CC)$.
Abelian representations are reducible, but the converse does not hold. Non-reducible representations are \emph{irreducible}.

Two representations $\rho$ and $\rho'$ in $R(\Gamma)$ are equivalent if they have the same trace:
\begin{equation*}
	\rho \sim \rho'  \text{ if and only if } \Tr \rho(\gamma) = \Tr \rho'(\gamma) \text{, for any }\gamma \in \Gamma.
\end{equation*}
The set of equivalence classes of representations coincides with the algebro-geometric  quotient of $R(\Gamma)$ by the action of $\SL_2(\CC)$ by conjugation. This quotient is usually constructed through invariant theory, and is denoted
\[X(\Gamma) = R(\Gamma)/ \! / \SL_2(\CC).\]
Points of the character variety are called \emph{characters}. The equivalence class of a representation $\rho$ (the character of $\rho)$ is denoted by $\chi_\rho : \Gamma \rightarrow \CC$ with $\chi_\rho(\gamma)=\Tr(\rho(\gamma))$ for $\gamma \in \Gamma$.
If $\Gamma$ is the fundamental group of a manifold $W$, we simply write $R(W)$ and $X(W)$ for the representation and character varieties of the manifold $W$.

Despite being abelian is not a well-defined notion on the character variety, the notion of being reducible makes sense there, since a reducible representation $\rho \colon \Gamma \to \SL_2(\CC)$ can be characterized by the fact that for any $\gamma, \delta \in \Gamma$, the following equality holds (see for instance \cite[Lemma 1.2.1]{CS83}):
\begin{equation}
	\label{eq:RedTrace}
	\Tr \rho\left(\gamma\delta\gamma^{-1}\delta^{-1}\right)=2.
\end{equation}

The character variety $X(\Gamma)$ can be decomposed as
\[X(\Gamma)= X^{\text{irr}}(\Gamma) \cup X^{\text{red}}(\Gamma),\]
where $X^{\text{red}}(\Gamma)$ is the set of reducible characters, and its complement $X^{\text{irr}}(\Gamma)$ is the set of irreducible characters. \cref{eq:RedTrace} implies that $X^{\text{red}}(\Gamma)$ is a Zariski closed subset of $X(\Gamma)$.

An algebraic set is \emph{reducible} if it can be written as a union of two proper algebraic subset, else it is \emph{irreducible}. An irreducible component of an algebraic set is a maximal irreducible algebraic subset.
\begin{remark}
	Despite $R(\Gamma)$ or $X(\Gamma)$ are called \emph{varieties}, they are not quite algebraic varieties in general: they are actually reducible, and might not be reduced as schemes (some points or subspaces might have multiplicity). On the other hand, any irreducible component is irreducible, and in particular reduced, by definition.
\end{remark}

Two representations $\rho$ and $\rho'$ are \emph{conjugate} if there exists a matrix $M \in \SL_2(\mathbb{C})$ such that $\rho(\gamma)=M \rho'(\gamma) M^{-1}$ for every $\gamma \in \Gamma$. Two conjugate representations define the same character; the converse is false in general, but true for elements of $X^{\text{irr}}(\Gamma)$. More precisely, the following holds.

\begin{proposition} \cite[Proposition 1.5.2]{CS83}
	\label{prop:conju}
	If $\rho$ and $\rho'$ are two representations $\GammaÂ \to \SL_2(\CC)$ with $\rho$ irreducible and $\chi_\rho=\chi_{\rho'}$, then $\rho$ and $\rho'$ are conjugate (and $\rho'$ is irreducible as well).
\end{proposition}
Two non-conjugate representations having the same character in $X(\Gamma)$ must be reducible. If $\Gamma$ is a knot group, Burde and De Rham \cite{Bur, dR67} showed that the set of characters containing non-conjugate representations is finite.

\subsection{The character variety of $\ZZ^2$}
\label{caractorus}
We describe explicitly the character variety of a 2-torus $S^1 \times S^1$.
Pick a basis $m,\ell$ of  $\pi_1(S^1 \times S^1) = \ZZ^2$.
Any representation in $\SL_2(\CC)$ is conjugate to a representation $\rho$ given by two commuting matrices of the form
\[\rho(m)= \bma M & * \\ 0 & M^{-1} \ema, \quad \rho(\ell) = \bma L & * \\ 0 & L^{-1} \ema \]
 for  $M,L \in \CC^*$. Each point of the character variety $X(S^1 \times S^1)$ has a pre-image in $R(S^1 \times S^1)$ of the form
\begin{equation}
	\label{equa:repbord}
	\rho(m)= \bma M & 0 \\ 0 & M^{-1} \ema, \quad \rho(\ell) = \bma L & 0 \\ 0 & L^{-1} \ema, \quad M,L \in \CC^*.
\end{equation}
This pre-image is unique up to the involution $\sigma$ of $ (\CC^*)^2$ sending $(L,M)$ to $(L^{-1},M^{-1})$, and $X(S^1 \times S^1)$ can be identified with the singular affine complex surface $(\CC^*)^2/\sigma$. It embedds in $\CC^3$ as the zeros of the polynomial
\[\Delta = x^2+y^2+z^2-xyz-4.\]
Indeed, the function algebra of $X(S^1 \times S^1)$ naturally identifies with  the $\sigma$-invariant sub-algebra $\CC[M+M^{-1}, L+L^{-1}]$ of $\CC[L^{\pm1}, M^{\pm1}]$. This algebra of invariant functions is isomorphic with $\CC[x,y,z]/(\Delta)$ through $M+M^{-1} \mapsto x, L+L^{-1} \mapsto y, ML+(ML)^{-1} \mapsto z$. From this description, one sees that the singular locus of $X(S^1 \times S^1)$ consists on the four points $\{ (L,M) = (\pm1 , \pm 1)\}$.

\subsection{The A-polynomial}
\label{subsec:Apol}

Let $K$ be an oriented knot in $S^3$ with exterior $\cM$. 
The inclusion $\partial \cM \subset \cM$ induces an injective group homomorphism $\pi_1(\partial \cM) \hookrightarrow \pi_1(\cM)$.
Let $r$ be the restriction map:
\[r \colon X(\cM) \longrightarrow X(\partial \cM) \simeq X(S^1 \times S^1).\]
For short we denote by $\rho_{\partial}= r(\rho)$ the restriction of $\rho$ to $\pi_1(\partial \cM)$.
By \cref{caractorus}, the choice of the longitude $\ell$ and the meridian $m$ induces an identification of $X(S^1 \times S^1)$ with a quotient of $(\CC^*)^2$.
The image of  $r$ is a union of points and curves
, possibly with multiplicities, see for instance \cite[Lemma 2.1]{Dunfield_Garouf}. Discarding the $0$-dimensional components, the \emph{A-polynomial} of $K$ is the unique polynomial $A_K(L,M)$ in {$\CC[L,M]$ } whose zero locus in $\CC^2$ is exactly mapped onto the image of
$r$. 
Note that  $A_K(L,M)$ is  always divisible by $L-1$. This factor corresponds to the curve of reducible characters. Boyer, Zhang, Dunfield and Garoufalidis have shown the following result.
\begin{theorem}[\cite{Boyer_Zhang, Dunfield_Garouf}]
	Let $K$ be a knot in $S^3$. The A-polynomial $A_K(L,M)$ is equal to $(L-1)^k$ for some $k$, if and only if $K$ is the trivial knot (and in this case $k=1$).
\end{theorem}

\section{The Slope invariant}
\label{sec:slope}
In this section we define the slope of an admissible representation (\cref{subsec:admi}), and observe that generic $\SL_2(\CC)$-representations are admissible. In \cref{subsec:slopechar} we show that the slope is invariant by conjugation of the representation. We prove in \cref{subsec:reg} that it yields a rational function on irreducible components of the character variety 
and that the slope of a real representation is a real number.
Then we prove in \cref{subsec:tors} that the slope can be written as a quotient of Reidemeister torsions. Finally, in \cref{subsec:Alex} we describe a procedure to compute the slope with an Alexander matrix.

\subsection{Admissible representations}
\label{subsec:admi}
Let $V$ be a finite dimensional $\CC$-vector space, and $\rho \colon \pi_1(\cM) \to \GL(V)$ be a representation. The representation extends to a ring homomorphism and $V$ can be viewed as a right $\ZZ[\pi_1(\cM)]$-module $V_\rho$. The twisted homology $H_*(\cM, \rho)$ is the homology of the complex of $\CC$-vector spaces:
\[C_*(\cM;\rho) = V_\rho \otimes_{\ZZ[\pi_1(\cM)]} C_*(\cM; \ZZ[\pi_1(\cM)]).\]
\begin{definition}
A representation $\rho \colon \pi_1(\cM) \to \GL(V)$  is \emph{admissible} if it satisfies:
\begin{itemize}[label=-]
	\item there exists $v_\rho \in V$
	such that $\{v_\rho \otimes \ell, v_\rho\otimes m\}$ is a basis of $H_1(\partial \cM, V_\rho) \simeq \CC^2$,
	\item the kernel of the homomorphism induced by inclusion:
	\[\mathcal Z(K,\rho) = \ker \left( H_1(\partial \cM, V_\rho) \stackrel{i_*} \longrightarrow H_1(\cM, V_\rho)\right)\]
	has dimension one.
\end{itemize}
\end{definition}

We restrict to representations $\rho \colon \pi_1(\cM) \rightarrow \SL_2(\CC)$. The composition of $\rho$ with the adjoint action $\Ad$ of $\SL_2(\CC)$ on $\mathfrak{sl}_2(\CC)$ induces the following representation:
\[\begin{array}{rccc}
	\Adr : & \pi_1(\cM) & \longrightarrow & \mathrm{Aut}(\mathfrak{sl}_2(\CC)) \\
	& \gamma & \longmapsto &\left(v \mapsto \rho(\gamma)v \rho(\gamma)^{-1} \right).
\end{array}\]

\begin{definition} \label{def:slope}
	Let $\rho \colon \pi_1(\cM) \to \SL_2(\CC)$ be such that $\Adr$ is admissible. Let $a \, (v_\rho \otimes \ell) + b \, (v_\rho \otimes m)$ be a generator of $\mathcal Z(K, \Adr)$ for some $[a:b] \in \CC \mathbb P^1$ .
	The \emph{slope} of the knot $K$ at the representation $\rho$ is
	\[s_K(\rho) =- \frac b a \in \CC \cup \infty. \]
\end{definition}

\begin{definition}
	A representation $\rho \colon \pi_1(\cM) \to \SL_2(\CC)$ is  \emph{boundary-parabolic} if the restriction
	$\rho_\partial: \pi_1(\partial \cM) \to \SL_2(\CC)$
	is parabolic, that is $\Tr \rho(\gamma) = \pm 2$ for any $\gamma \in \pi_1(\partial \cM)$.

	A boundary-parabolic character is the character of a boundary-parabolic representation.
\end{definition}

\begin{lemma}
	\label{lem:dim1}
	Let $\rho\colon \pi_1(\cM) \to \SL_2(\CC)$ be a non-parabolic representation. The vector space $H_1(\partial \cM, \Adr)$ is isomorphic to $\CC^2$, and the kernel $\mathcal Z(M, \Adr)$ is one-dimensional. Moreover, if $\rho$ is not boundary-parabolic, then $\Adr$  is admissible.
\end{lemma}

\begin{proof}
	The group $\pi_1(\cM)$ is generated by 
	pairwise conjugate meridians. 
	If $\rho$ is non-parabolic, then the image of a meridian must differ to $\pm I_2$, otherwise we would have $\rho(\pi_1(\cM)) \subset \{ \pm I_2 \}$.

	Consider the complex $C_*(\partial \cM, \Adr)$ with one $0$-cell, two $1$-cells corresponding to $\ell$ and $m$ and one $2$-cell.
	An explicit computation of the homology of shows that the dimension of $H_1(\partial \cM, \Adr)$ is two. Moreover, when $\rho$ is not boundary-parabolic, for $v_\rho \in \mathfrak{sl}_2(\CC)$  invariant by $\Adr_\partial$, the pair of vectors
	$(v_\rho \otimes \ell,v_\rho \otimes m)$ forms a basis of $H_1(\partial \cM, \Adr) $.
	Since the Killing form on the local system $\mathfrak {sl}_2(\CC)$ yields a Poincar\'e duality isomorphism $H_*(\cM, \Adr) \simeq H_{3-*}(\cM, \partial \cM, \Adr)$, the following diagram is commutative
	\[\begin{tikzcd}
		H_2(M, \partial \cM, \Adr) \rar["`d"] \dar & H_1(\partial \cM, \Adr) \rar["i_*"] \dar & H_1(\cM, \Adr) \dar \\
		H_1(\cM,\Adr)^\ast \rar["i^*"] & H_1(\partial \cM, \Adr)^\ast \rar["`d^*"] & H_2(\cM, \partial \cM, \Adr)^\ast
	\end{tikzcd}\]
	with exact rows and where the vertical arrows are isomorphisms. Exactness implies
	\[\dim \ker i_* = \operatorname{rank} \delta=\operatorname{rank} i^*=\operatorname{rank} i_*\]
	since the diagram commutes and transposition preserves the rank. Hence $\dim \mathcal Z(K, \Adr)=\dim \ker i_* =  (1/2) \dim H_1(\partial M, \Adr).$
\end{proof}

As an example, we compute the slope for abelian non boundary-parabolic representations.
Let $\varphi \colon \pi_1(\cM) \to H_1(\cM) = \ZZ$ be the abelianization.
For any $\lambda \in \CC^*$, there is an abelian representation
\begin{align}
	\label{abel}
	\rho_\lambda \colon  \pi_1(\cM) & \longrightarrow  \SL_2(\CC)\\
	 \gamma & \longmapsto  \bsm \lambda^{\varphi(\gamma)} & 0 \\ 0 & \lambda^{-\varphi(\gamma)}\esm \nonumber
\end{align}
and any abelian, non boundary-parabolic representation is conjugate to a representation of this form.

\begin{lemma}
	\label{lem:abel}
	For any $\lambda \neq \pm 1$, the slope at the abelian representation $\rho_\lambda$ vanishes:
	\[s_K(\rho_\lambda)=0.\]
\end{lemma}
\begin{proof}
	Up to conjugation, the representation $\Adr$ has the form
	\[\Adr(\gamma) =  \bsm \lambda^{2\varphi(\gamma)} &0 &0\\0& 1 &0 \\0 &0 &\lambda^{-2\varphi(\gamma)}\esm\]
	and $\mathfrak{sl}_2(\CC)$ splits as $\ZZ[\pi_1(\cM)]$-module as
	\[\mathfrak{sl}_2(\CC) = \CC_{\lambda^2} \oplus \CC \oplus \CC_{\lambda^{-2}}.\]
	This yields a splitting in twisted homology (with abelian coefficients), for $U= \partial \cM$ or $U=\cM$:
	\[H_1(U, \Adr) = H_1\left(U,  \CC_{\lambda^2}\right) \oplus  H_1(U,\CC) \oplus H_1\left(U,  \CC_{\lambda^{-2}}\right)\]
	Since $\lambda\neq \pm 1$, for $U=\partial \cM$ the only non-trivial summand is $H_1(\partial \cM, \CC)$, and the map $H_1(\partial \cM, \Adr) \xrightarrow{i_*} H_1(\cM, \Adr)$ coincides with the map induced by the inclusion in homology with trivial coefficients $H_1(\partial \cM, \CC) \to H_1(\cM, \CC)$, whose kernel is generated by~$\ell$.
\end{proof}

\subsection{The slope of characters}
\label{subsec:slopechar}
By the following lemma, the slope does not depend on the conjugacy class of an irreducible representation. 
Combined with \cref{prop:conju}, it follows that the slope of an irreducible representation depends only on its character.

\begin{lemma}
	\label{lem:conj}
	Let $\rho$ and $\rho' \colon \pi_1(\cM) \to \SL_2(\CC)$ be two irreducible, non boundary-parabolic representations. If $\rho$ and $\rho'$ are conjugate, then $s_K(\rho)= s_K(\rho')$.
\end{lemma}

\begin{proof}
	Let $A$ be a matrix in $\GL_2(\CC)$ such that $\rho'= A\, \rho \, A^{-1}$. Any $\Adr$-invariant vector $v_\rhoÂ \in \mathfrak{sl}_2(\CC)$ yields an $\Adr'$-invariant vector $v'_\rho = A\,v_\rho\,A^{-1}$, and the conjugation by $A$ induces an isomorphism
	\[H_1(\partial \cM, \Adr) \to H_1(\partial \cM, \Adr')\] sending the basis $\{v_\rho \otimes \ell, v_\rho\otimes m\}$ to $\{v'_\rho `(`*) \ell, v'_\rho `(`*) m\}$ and the subspace $\mathcal Z(K, \Adr)$ to $\mathcal Z(K, \Adr')$. Hence
	$s_K(\rho)= s_K(\rho')$.
\end{proof}

\begin{remark}
	\label{remk:abel} 
	There exist pairs of reducible, non-conjugate representations with the same character. Indeed, let $\chi$ be an arbitrary reducible character in $X(\cM)$. Consider a representation $\rho$ of the form  $\bsm \lambda(\gamma) & * \\ 0 & \lambda^{-1}(\gamma)\esm$, where $\lambda \colon \pi_1(\cM) \to \CC^*$ is a group homomorphism, choosen such that $\chi(\rho)=\chi$.
	Note that $\lambda$ can further be written $\lambda(\gamma) = \lambda^{\varphi(\gamma)}$ for some $\lambda \in \CC^*$ and $\varphi: \pi_1(\cM) \to \ZZ$.
	Hence the abelian representation $\rho_\lambda$ defined in \cref{abel} has also character $\chi$, but is not conjugated in general to $\rho$.
	It turns out that they can have different slope values.

	For example, consider the right-handed trefoil knot $T$ in $S^3$. The character variety $X(M_T)$ is the union of a line $X^{\text{red}}$ and a conic $X^{\text{irr}}$ in the plane. The line contains only reducible characters, and any character in the conic is irreducible except the two intersection points $X^{\text{red}} \cap X^{\text{irr}}$. Let $\chi$ be a point in $X^{\text{red}} \cap X^{\text{irr}}$. Since $\chi$ is reducible, there exist $\lambda \in \CC^*$ such that the abelian representation $\rho_\lambda$ has character $\chi$. By \cref{lem:abel}, one has $s_T(\rho_\lambda) = 0$. However, we show in \cref{ex:trefoil} that the slope defines a constant function on $X^{\text{irr}}$, everywhere equal to $-6$.
\end{remark}
\subsection{Regularity and properties of the slope}
\label{subsec:reg}

We extend the slope to a rational function --locally a quotient of polynomials-- on the character variety $X(\cM)$. 

\medbreak

There is a component $X^{\text{red}}\subset X(\cM)$ 
of reducible characters only. By \cref{remk:abel} any character in $X^{\text{red}}$ is the character of an abelian representation. Hence the slope is identically zero on $X^{\text{red}}$, see \cref{lem:abel}.
Suppose now that $X\subset X(M)$ is an irreducible component containing an irreducible character.

\begin{proposition}
	\label{cor:slExt}
	Let $X\subset X(M)$ be an irreducible component which contains an irreducible  character. The slope extends to a rational function on $X$, still denoted $s_K$.
		Moreover, if  $\chi \in X$ is a boundary-parabolic character
	then
	\begin{equation}
		\label{slopemodule}
		s_K(\chi) = \tau(\chi),
	\end{equation}
	where the modulus $\tau(\chi) \in \CC$
	is defined by taking the representative $\rho$ of $\chi$ satisfying
	\begin{equation*}
		\rho(m)=  \bma \pm1&1 \\0 & \pm1 \ema, \quad \rho(\ell) =   \bma \pm 1& \tau(\chi) \\0 &\pm 1 \ema.
	\end{equation*}
\end{proposition}


\begin{remark}
	If $\chi$ is the character of an irreducible representation and lies at the intersection of several {irreducible} components, then the value of the slope at $\chi$ is well-defined.
\end{remark}
The rest of the section is devoted to the proof of \cref{cor:slExt}.
\cref{lem:rational} asserts that the slope is a rational function in the neighborhood of any irreducible, non boundary-parabolic character.
For boundary parabolic characters $\chi$, we \emph{define} the slope by the relation in \cref{slopemodule} and we show that the result is still a rational function on $X$ in \cref{prop:cusp}.

\begin{lemma}
	\label{lem:rational}
	Let $\chi_0$ an irreducible, non-boundary-parabolic character in $X$. The slope is a rational function in a neighborhood of $\chi_{0}$ in $X$.
\end{lemma}

\begin{proof}
	Let $\rho_0$ in $R(\cM)$ be a representation with character $\chi_0$. 
	By \cref{lem:dim1} one has $H_1(\partial \cM, \Adr) \simeq \CC^2$. The set of complex lines
	\[\mathbb P(\rho) = \mathbb P(H_1(\partial \cM, \Adr))\]
	is a complex algebraic variety isomorphic to $\CC P^1$. If $\rho$ and $\rho'$ are conjugate, then there is a natural algebraic isomorphism $\mathbb P(\rho) \simeq \mathbb P(\rho')$. It defines an algebraic $\CC\mathbb P^1$-fibration on a neighborhood of $\chi_{\rho_0}$, and  for any $\chi$, the complex line $\mathcal Z(\cM, \Adr)$ is an algebraic section of this fibration, independent of the choice of representation $\rho$ with character $\chi$.

	It remains to show that the identification $\mathbb P(\rho) \simeq \CC\mathbb P^1$ is algebraic, in other words, that the choice of the basis $(v_\rho \otimes \ell, v_\rho \otimes m)$ depends algebraically on $\rho$.
	Since $\rho_0$ is not boundary-parabolic, we can shrink the chosen neighborhood so that no representation $\rho$ near $\rho_0$ is boundary-parabolic. Then, since $\rho_\partial$ is conjugated to a diagonal representation, there is a unique $\Adr_\partial$-invariant vector $v_\rho$ with norm $1$  in $\mathfrak{sl}_2(\CC)$.
	This choice depends polynomially on the entries of the matrix $\Adr(m)$, and then the basis $(v_\rho \otimes \ell, v_\rho \otimes m)$ depends algebraically on $\rho$.
\end{proof}

We now consider the case of boundary-parabolic characters.
\begin{lemma}
	\label{lem:bp}
	Let $\rho_0$ be a boundary-parabolic representation whose character $\chi_{\rho_0}$ lies in $X$. Then $\rho_0$ is irreducible, in particular $\rho_0(m)\neq \pm I_2$.
\end{lemma}
\begin{proof}
	For $\rho$ reducible in $X$, \cite{Bur, dR67} implies that $\rho(m)$ has eigenvalues $\lambda, \lambda^{-1}$ in $\CC$, whose square is a root of the Alexander polynomial $\Delta_{M_K}(t)$, in particular $\lambda\neq \pm 1$, and $\rho$ is not boundary-parabolic. Now for irreducible $\rho$, the image of any meridian must be different of  $I_2$, since meridians generate the group $\pi_1(\cM)$.
\end{proof}

\begin{lemma}
	\label{prop:cusp}
	Let $X\subset X(M)$ be an irreducible component containing an irreducible character, and $\chi_0 \in X$ a boundary-parabolic character.
		Then the slope function $s_K$ is rational in a neighborhood of $\chi_0$.
\end{lemma}

\begin{proof}
	Suppose first that $X$ contains only boundary-parabolic characters. Any $\chi \in X$ is the character of a representation $\rho$ such that $\rho(m) = \bsm \pm 1 & 1 \\ 0 & \pm 1 \esm$. Hence $\chi \mapsto \tau(\chi)$ is rational.

	Now we assume that $X$ contains a non boundary-parabolic character. By definition, boundary-parabolic characters form a Zariski closed subset of $X$. By \cref{lem:rational}, the slope function is rational on the open, non-empty subset of $X$ consisting of non boundary-parabolic characters. By analytic continuation, it is enough to show that
	\[\lim_{\chi \to \chi_0}Â s_K(\chi) = \tau(\chi_0).\]

	By \cref{lem:bp} any boundary-parabolic representation $\rho_0$ with character $\chi_0$ is irreducible. Moreover, since $\rho_0(m)$ can not be trivial, we can chose such a $\rho_0$ satifying
	\[\rho_0(m) = \bsm \pm 1 & 1 \\ 0 & \pm 1\esm.\]
	For any $\chi$ close to $\chi_0$, we chose similarly a representation $\rho$ with character $\chi$ such that
	\[\rho(m) = \bsm M & 1 \\ 0 & M^{-1} \esm,\] 
	with $M$ close to $\pm 1$ in $\CC^*$.
		For such $\rho$, let $v_\rho = \bsm M-M^{-1} & 2 \\ 0 & M^{-1} - M\esm $ be an $\Adr_{\partial}$-invariant vector. The limit at $\rho_0$ of $v_\rho$ is the $(\Adr_0)_\partial$-invariant vector $v_{\rho_0} = \bsm 0 & 2 \\ 0 & 0 \esm$. However, a direct computation shows that $v_{\rho_0} \otimes \ell$ and $v_{\rho_0} \otimes m$ are linearly dependent in $H_1(\partial \cM, \Adr_0)$, and we cannot compute the slope of the boundary parabolic representation $\rho_0$ by means of \cref{def:slope}. Nevertheless, the subspace $\mathcal Z(K,\Adr_0)$ is one-dimensional by \cref{lem:dim1}.

	It implies that the map $i_* \colon H_1(\partial \cM, \Adr) \to H_1(\cM, \Adr)$ has rank one at any representation $\rho$ near $\rho_0$, and at $\rho_0$ as well. In particular, for any $\rho$ near $\rho_0$, the slope can be computed as
	{the ratio of $i_*(v_\rho \otimes \ell)$ and $i_*( v_\rho \otimes m)$ in $i_*(H_1(\partial \cM, \Adr))$. This actually makes sense for $\rho = \rho_0$ as well.
	An explicit computation of the boundary operator $\partial_1 \colon C_2(\partial \cM, \Adr_0) \to C_1(\partial \cM, \Adr_0)$ shows that the vector $ v_{\rho_0} \otimes \ell - \tau(\chi_0) \, v_{\rho_0} \otimes m$ belongs to $\operatorname{im} \partial_2$, and the equality
	\[v_{\rho_0} \otimes \ell  =\tau(\chi_0)\,  v_{\rho_0} \otimes m\]
	holds 
	in $H_1(\partial \cM, \Adr_0)$. This implies that the ratio of $i_*(v_{\rho_0} \otimes \ell )$ and $i_*(v_{\rho_0} \otimes m )$ coincides with the modulus $\tau(\chi_0)$. This proves the lemma, and achieves the proof of \cref{cor:slExt}.
	}
\end{proof}

We end up this section with the following observation.
\begin{proposition}
\label{prop:real}
	Let $X\subset X(M)$ be an irreducible component which contains a non boundary-parabolic representation.
	If $\rho \in X$ is a \emph{real}  representation $\rho \colon \pi_1(\cM) \to \SL_2(\RR)$ or $\rho \colon \pi_1(\cM) \to \SU(2)$, then the slope is a real number in $\RR\mathbb P^1$.
\end{proposition}

\begin{proof}
	First assume that $\rho$ is non-boundary-parabolic. If $\rho$ is real, denoting by $\Adr_{\RR}$ the action of $\rho$ on the Lie algebra $\mathfrak{sl}_2(\RR)$ (resp. $\mathfrak{su}(2)$) of $\SL_2(\RR)$ (resp. $\SU(2)$), then obviously the Lagrangian $\mathcal Z(\cM, \Adr) \subset H_1(\partial \cM, \Adr)$ is the complexification of the real Lagrangian $\mathcal Z(\cM, \Adr_\RR)$ in the real symplectic vector space $H_1(\partial \cM, \Adr_\RR)$ and the slope of this real Lagrangian is the slope of its complexification, a real number. If $\rho$ is boundary-parabolic and re	al, then it takes value into $\SL_2(\RR)$ and the proposition follows from the definition of the modulus $\tau$. \end{proof}

\subsection{Slope and Reidemeister torsion}
\label{subsec:tors}
In this section we show that the slope coincides with the "change of curve term" for the Reidemeister torsion as stated in \cref{prop:collTors}.

\medbreak

If $\rho$ is an irreducible representation in $X(\cM)$, we consider the torsion of the complex $C_*(\cM, \Adr)$ defined in  \cref{subsec:admi}. This complex is naturally based from a cell  decomposition of $\cM$ and a choice of a basis of $\mathfrak{sl}_2(\CC)$, but not acyclic. The Reidemeister torsion is usually defined for acyclic complexes. In the case we are considering, one needs to make some additional choices to define it, namely a basis of each homology group $H_*(\cM, \Adr)$.

According to \cite{Porti97}, one can still define the Reidemeister torsion of the cellular complex $C_*(\cM, \Adr)$ for representations $\rho$ in $R(\cM)$ such that $H_1(\cM, \Adr)$ has dimension~$1$.  For a given curve $\gamma \in \pi_1(\partial \cM)$, the representation $\rho$ is \emph{$\gamma$-regular} if there exists a vector $v_\rho \in \mathfrak{sl}_2(\CC)$ such that $v_\rho \otimes \gamma$ spans $H_1(\cM, \Adr)$.
In this case, since there is a natural choice of a basis of $H_2(\cM, \Adr)$, the curve $\gamma$ determines a homology basis of the complex $C_*(\cM, \Adr)$ and the torsion $\mathbb{T}_{\cM, \gamma}(\Adr) \in \CC^*$ is defined. Note that this torsion depends only on the conjugacy class of $\rho$, as well as the property of being $\gamma$-regular.

Let $X \subset X(\cM)$ the component containing $\chi$, the torsion function {is the rational function}
\[\mathbb T_{\cM, \gamma} \colon X \to \CC\]
defined as the Reidemeister torsion of the complex $C_*(\cM, \Ad)$ if $\chi$ is $\gamma$-regular, and by $T_{\cM, \gamma} (\chi) = 0$ 
otherwise.

We start with the following lemma, which 
provides the genuine setting to define the Reidemeister torsion.

\begin{lemma}
	\label{lem:dim}
	If $X$ has dimension one and contains the character of a scheme-smooth representation $\rho$, then $\dim H_1(\cM, \Adr)=1$.
\end{lemma}
\begin{proof}
	The proof of \cref{lem:dim} follows from the isomorphism between $H^1(\cM, \Adr)$ and the Zariski tangent space of $X(\cM)$ at $\rho$, see  \cite[Theorem 1]{Sikora} . Scheme-smoothness implies that the Zariski tangent space is the actual tangent space, which is one-dimensional because $X$ is.
\end{proof}
Note that scheme-smoothness is a Zariski open condition.

It turns out that the character variety $X(\cM)$ of a knot exterior is often one-dimensional. This is the case if the knot is \emph{small} (if it does not contains a closed incompressible surface \cite[Proposition 2.4]{CCGLS}). This is also the case for any component $X \subset X(\cM)$ containing the character of a lift of the holonomy representation $\overline \rho \colon \pi_1(\cM) \to \operatorname{PSL}_2(\CC)$, provided that the interior of $\cM$ admits a hyperbolic structure.

The following proposition is the main result of this section.
\begin{proposition}
	\label{prop:slopetors}
	Let $X \subset X(M)$ be an irreducible one-dimensional component which contains a scheme-smooth, non-boundary parabolic character.  For all $\chi \in X$ the following holds
	\[s_K(\chi) = \frac{\mathbb{T}_{\cM,\ell}(\chi)}{\mathbb{T}_{\cM,m}(\chi)}.\]
\end{proposition}
We provide two different proofs of this result: one uses the natural definition of the torsion while the other relies directly on some results on the \emph{torsion form} proved by the first author in \cite{Ben16}.

\subsubsection{Torsion and chain complexes}
This section is devoted to the proof of \cref{prop:slopetors} by using the chain complex of $\cM$. The proof is very similar to \cite[Theorem 3.21]{DFL1} or \cite[Theorem 6.7]{DFL1}.
We use the following technical lemma.
\begin{lemma}
	\label{lem:regopen}
	Let $\gamma$ be a curve in $\pi_1(\partial \cM)$, and $\chi$ be an irreducible $\gamma$-regular character in $X(\cM)$. There exists a Zariski open neighborhood of $\chi$ such that any character in this neighborhood is irreducible and $\gamma$-regular.
\end{lemma}
\begin{proof}
	Being irreducible is a Zariski open condition, see \cref{eq:RedTrace}. The $\gamma$-regularity follows from lower semi-continuity of the rank of a linear map. Indeed the dimension of $H_1(\cM, \Adr)$ is upper semi-continuous. It is at least one (the dimension of $X$) again because it is isomorphic to the Zariski tangent space hence it is locally constant equal to one.
	On the other hand, the rank of the linear map $H_1(\gamma, \Adr) \to H_1(\cM, \Adr)$ sending $v_\rho \otimes \gamma$ to itself is lower semi-continuous. It is at most one (the dimension of $H_1(\gamma, \Adr)$ and it cannot decrease on a neighborhood of $\chi$.
	We deduce that $H_1(\gamma, \Adr) \to H_1(\cM, \Adr)$ is an isomorphism on a Zariski open subset.
\end{proof}

\begin{proof}[Proof of \cref{prop:slopetors}]
	Let $\chi$ be an irreducible, scheme-smooth,  non boundary-parabolic character and let $\rho$ be a representation in $R(\cM)$ with character $\chi$.
	We first assume that $\rho$ is $\ell$ and $m$-regular, that is for $v \in \mathfrak{sl}_2(\CC)$ an $\Adr_\partial$-invariant vector,  both $v \otimes \ell$ and $v \otimes m$ provide a basis of $H_1(\cM, \Adr)$.

	The calculation of the torsions $\mathbb{T}_{\cM,\ell}(\chi)$ and $\mathbb{T}_{\cM,m}(\chi)$ involves different choices of homology basis of $C_*(\cM, \Adr)$. By \cite[Proposition 3.18]{Porti97},  the bases of $H_2(\cM, \Adr)$ are determined by the fundamental class of $H_2(\partial \cM; \CC)$ and can be choosen to be the same.
	Hence,  if $b_1$ is a basis of $\mathrm{im} (\partial_1)$, the ratio of torsions corresponding to the choice of $m$ or of $\ell$ is {reduced to}
	\begin{equation*}
		\frac{\mathbb{T}_{\cM,\ell}(\chi)}{\mathbb{T}_{\cM,m}(\chi)}= \frac{\det(b_1 \oplus (v \otimes \ell), c_1)}{\det(b_1 \oplus (v \otimes m),c_1)}.
	\end{equation*}
	In parallel, consider the affine equation in $C_1(\cM, \Adr)$:
	\[y \, b_1 + x \, v \otimes m= v \otimes \ell,\]
	 with at least a solution $y=0$ and $x= s_K(\rho)$.
	 The Cramer determinants expressed in the common basis $c_1$ show that $s_K(\rho)$ coincides with the ratio of torsions.

	If there exists a character in $X$ which is $\ell$-regular and a character in $X$ which is $m$-regular, then \cref{prop:slopetors} holds on the whole component $X$ by \cref{lem:regopen}.

	Assume that $X$ contains only characters that are not (say) $\ell$-regular. Since the map $H_1(\partial \cM, \Adr) \to H_1(\cM, \Adr)$ is not trivial (by \cref{lem:dim1}), it is onto on a Zariski open subset $U \in X$, again because $H_1(\cM, \Adr)$ has dimension one generically. Thus all characters in $U$ must be $m$-regular, and it follows from the definition that the slope and the quotient of torsions are identically zero on $X$. A similar argument works replacing $\ell$ by $m$ and zero by infinity.
\end{proof}

\subsubsection{The torsion form} \label{torsionform}
In this paragraph, we present an alternative proof of \cref{prop:slopetors}. We follow a slightly different point of view on the torsion, as a volume form on the character variety. The following lemma asserts that the cotangent space of the character variety \cite[Section 8]{Sikora} is isomorphic to the first $\Adr$-twisted homology group.
\begin{lemma}
	\label{prop:Cotangent}
	Let $\chi$ be an irreducible character in $X(\cM)$, and a representation $\rho$ with character $\chi$. Let $T^*_{\chi} X(\cM)$ be the Zariski tangent space
	of $X(\cM)$ at $\chi$. There is a natural isomorphism
	\[H_1(\cM, \Adr) \simeq T^*_{\chi} X(\cM).\]
	Moreover, if $\chi$ is not boundary-parabolic, then
	\[ H_1(\partial \cM, \Adr) \simeq T^*_{r(\chi)} X(\partial \cM).\]
\end{lemma}

The proof of \cref{prop:Cotangent} follows from \cite[Theorem 1]{Sikora}. Note that through the isomorphim, the space $\mathcal Z(K, \Adr)$ is the Zariski conormal bundle of $r(X(\cM))$ in $X(\partial \cM)$.

If $X \subset X(\cM)$ is a one-dimensional component of the character variety which contains a scheme-smooth character,  the first author proved in \cite[Proposition 5.1]{Ben16} that the \emph{torsion form} can be written as
\begin{equation}
	\label{eq:torsform}
	\tor(\cM) = \frac 1 {{\mathbb{T}}_{\cM, \ell}} r^* \left( \frac {dL}{L} \right)= \frac 1 {{\mathbb{T}}_{\cM, m}} r^*\left( \frac {dM}{M}\right)
\end{equation}
where $r^*$ is the cotangent map
\[r^*\colon T^* X(\partial \cM)\to T^*X(\cM).\]

\begin{proof}[Proof of \cref{prop:slopetors}]
	By \cref{eq:torsform}, the ratio of torsions can be written as
	\[\frac{{\mathbb{T}}_{\cM, \ell}}{{\mathbb{T}}_{\cM, m}} = \frac {r^*( dL/L)}{r^*(dM/M)}.\]
		If  $\chi$ is a non boundary-parabolic character, the character variety $X(\partial \cM)$ is diffeomorphic to $(\CC^*)^2$ in a neighborhood of $r(\chi)$.  A local chart of $X(\partial \cM)$  is given by taking $\frak l, \frak m \in \CC$ satisfying $\exp \frak l = L$ and $\exp \frak m = M$.
	The latter ratio of torsions can be  written
	\[\frac{{\mathbb{T}}_{\cM, \ell}}{{\mathbb{T}}_{\cM, m}} =\frac {r^*(d\,\frak l)}{r^*(d\,\frak m)}.\]

		\cref{prop:Cotangent} implies that the cotangent map $r^*\colon T^*_{r(\chi)} X(\partial \cM)\to T^*_\chi X(\cM)$ coincides with the homomorphism in homology $H_1(\partial \cM, \Adr) \to H_1(\cM, \Adr)$, thus by \cref{lem:dim1} the range of the map $r^*$ is one-dimensional, and the images of the elements $d\, \frak l, d \, \frak m$ are colinear. It turns out that the ratio $\frac {r^*(d\,\frak l)}{r^*(d\,\frak m)}$ coincides with the slope by its very definition.

	Finally, the formula extends to the whole $X$ since irreducible, non boundary-parabolic character are Zariski dense in $X$.
\end{proof}
\subsection{Compute the slope}
\label{subsec:Alex}
In this section we compute the slope $s_K(\rho)$ when $\rho$ is an irreducible non-boundary parabolic representation, with Fox calculus, similarly to  \cite{DFL1}.
Note that for the boundary-parabolic case, the slope can be computed directly from the representation using \cref{cor:slExt}.

Consider a  presentation of the knot group
\[`p_1(\cM) = \left<x_1, \ldots, x_p \mid r_1, \ldots, r_q\right>\]
obtained from a Wirtinger presentation. We also assume that $m = x_1$ is the preferred meridian and add the preferred longitude $\ell = x_2$,  with the relation $[m,\ell]=1$. Consider the complex of $\ZZ[`p_1(\cM)]$-modules
\[S_* := S_2 "-{^{\partial_1}}>" S_1 "-{^{\partial_0}}>" S_0\]
where
\begin{equation*}
	S_2 = \bigoplus_{j=1}^q \ZZ[`p_1(\cM)] `(`*) r_j, \  S_1 = \bigoplus_{j=1}^p \ZZ[`p_1(\cM)] `(`*) d x_i , \ S_0 = \ZZ[`p_1(\cM)]
\end{equation*}
and $dx_i$ is a formal generator corresponding to $x_i$. Let $\partial / \partial x_i \colon \ZZ[\pi_1(\cM)] \to \ZZ[\pi_1(\cM)]$ denote the Fox derivatives.
For every $i `: \{1, \dots, p	\}$ and $j `: \{1, \dots, q\}$ the boundary operators  are defined by
\begin{equation*}
	\partial_1 : r_j "|->" dr_j ,\quad  \partial_0 : dx_i "|->" x_i,
\end{equation*}
where $dw$ is the Fox differential of the word $w `: `p_1(\cM)$:
\[dw := \sum_{i=1}^p \frac{\partial w}{\partial x_i}dx_i `: S_1.\]
Now consider the $\Adr$-twisted chain complex $S_*(`r) := \mathfrak{sl}_2(\mathbb{C}) \: {`(`*)}_{\ZZ[`p_1(\cM)]} \: S_*$, where elements
 of $\mathfrak{sl}_2$ are identified with line vectors in $\CC^3$.
The ($\Adr$-twisted) \emph{Alexander matrix} of $\cM$ associated with the presentation is the matrix of $\partial_1 (`r)$, {with coefficients in $\CC$} {and given by the blockwise definition}:
\[{\left((\Adr) \left(\dfrac{\partial r_i}{\partial x_j}\right)\right)}_{1 \leq i \leq q,\: 1\leq j \leq p}\]

The computation of the slope using $S_*(`r)$ is achieved with the following result:
\begin{proposition}
	\label{prop:slcomp}
	If $`r$ is irreducible and non-boundary parabolic, then there exists $a,b `: \CC$ and an $\Adr_\partial$-invariant vector $v_\rho \in \mathfrak{sl}_2(\CC)$  such that
	\[\mathrm{im}\left(\partial_1 (`r)\right) \cap \left< v_{`r} `(`*) d\ell, v_{`r} `(`*) dm \right> = \left< a \, (v_{`r} `(`*) d\ell) + b \, (v_{`r} `(`*) dm) \right>\]
	and the slope is $s_{K}(\rho) = -\frac{b}{a}$.
\end{proposition}
\begin{proof}
	Set a base point $p$ on $\partial \cM$. Following Crowell~\cite{Crowell71}, the homology space $H_1(S_*(`r))$ is isomorphic to $H_1(\cM,p,\Adr)$.
	The subcomplex ${S_*(`r_{\partial})}$ defined by considering only the generators $x_1=m$, $x_2=\ell$ and the relation $[m,\ell] = 1$ computes the space $H_1(\partial \cM,p;\Adr)$ as well. There are natural identifications 
	\[H_1\left(\partial \cM,p;\Adr \right)=H_1(\partial \cM,\Adr)  \text{ and }  H_1(\partial \cM,p;\Adr) \hookrightarrow H_1(\partial \cM,\Adr) \]
	and the following diagram commutes:
	\[\begin{tikzcd}
		{S_1(`r_{\partial})} \ar[d,"h_{\partial \cM}"',two heads] \ar[rr,hook] & & S_1(`r) \ar[d,"h",two heads] \\
		H_1(\partial \cM, \Adr) \ar[r,"i_*"] & H_1(\cM,\Adr) \ar[r,hook] & H_1(\cM,{p},\Adr)
	\end{tikzcd}\]
	where $h$ and $h_{\partial \cM}$ are the quotient maps.

	Let $u \in \mathfrak{sl}_2(\CC)$ be an $\Adr$-invariant vector and $\gamma \in \pi_1(\cM)$.  Any element $u `(`*) \gamma$ of  $H_1(\partial \cM, \Adr)$ can be lifted to  $u `(`*) dw$ in $S_1(`r_{\partial})$. Since $\rho$ is admissible, there exists $a,b `: \CC$ such that $\mathcal Z(K,\Adr)=\ker i_* = \left<a \, (v_{`r} `(`*) \ell) + b \, (v_{`r} `(`*) m)\right>$. Then $a \, (v_{`r} `(`*) {d\ell}) + b \, (v_{`r} `(`*) {dm}) `: \ker (h) = \mathrm{im}\left(\partial_1 (`r)\right)$.

	Reciprocally, suppose that there exists complex numbers $a,b `: \CC$ such that $dz := a \, (v_{`r} `(`*) {d\ell}) + b \, (v_{`r} `(`*) {dm})$ is a non-zero vector belonging to $\mathrm{im}\left(\partial_1 (`r)\right)$. Then  $h_{\partial \cM}(dz) = a \, (v_{`r} `(`*) \ell) + b \, (v_{`r} `(`*) {m})$  must be non-zero since $(v_{`r} `(`*) \ell, v_{`r} `(`*) {m})$ is a free basis of $H_1(\partial \cM, \Adr)$. However, $h(dz) = i_*\left(h_{\partial \cM}(dz)\right)=0$; hence $h_{\partial \cM}(dz) `: \ker i_*$. Since $\ker i_*$ is one-dimensional, then $\ker i_* = \left<h_{\partial \cM}(dz)\right>$, and the slope is 
 $-\frac{b}{a}$.
\end{proof}
\begin{example} \emph{The trefoil knot}
	\label{ex:trefoil}
	Let $T$ be the exterior of the right-handed trefoil knot, with group $\pi_1(\cM) = \langle u,v \mid uvu=vuv\rangle.$ Any irreducible representation is conjugate to $\rho$ with
	\begin{equation*}
		\rho(u) = \bma M & 1 \\0 & M^{-1} \ema ,\quad \rho(v) = \bma M^{-1} & 0 \\ -1 & M \ema
	\end{equation*}
	 where $M `: \CC$.
	If $\ell= vuv^{-1}uvu^{-3}$ is the prefered longitude with corresponding meridian $m=u$, we obtain
	\[\rho(\ell) = \bsm -M^{-6} & M^5+M^3 + M + M^{-1} +M^{-3} + M^{-5}\\ 0 & -M^6\esm.\]
	Whenever $M \neq \pm 1$, the vector $v_\rho = \bma 0, 1, \frac 1 {M-M^{-1}} \ema$ is right $\Adr_{\partial}$-invariant.
	By \cref{subsec:Alex}, the Alexander matrix (acting on the right on the coefficients) whose rowspace is generating $\mathrm{im}\left(\partial_1\right)$ is given by
		\[\begin{array}{rr*{3}{c}:*{3}{c}:*{3}{c}l}
		& & \multicolumn{3}{c}{\mathfrak{sl}_2(\mathbb{C}) `(`*) d\ell} & \multicolumn{3}{c}{\mathfrak{sl}_2(\mathbb{C}) `(`*) dm} & \multicolumn{3}{c}{\mathfrak{sl}_2(\mathbb{C}) `(`*) dv} & \\
		\multirow{3}{*}{$dr_1$} &\ldelim({6}{*} & 0 & 0 & 0 & 1 & 0 & 0 & -{M^{-2}} & 0 & 0 & \rdelim){6}{*}\\
		& & 0 & 0 & 0 & 0 & 1 & 0 & -M^{-1} & -1 & 0 & \\
		& & 0 & 0 & 0 & 0 & 0 & 1 & 1 & 2 \, M & -M^{2} & \\
		\cdashline{3-11}
		\multirow{3}{*}{$dr_2$} & & -1 & 0 & 0 & -2 {\left( 2 - M^{-2}\right)} & 0 & 0 & 1 + M^{-2} & 0 & 0 & \\
		& & 0 & -1 & 0 & 2{M}^{-1} & -2 & 0 & {M}^{-1} & 2 & 0 \\
		& & 0 & 0 & -1 & -2 & -4 \, M & 2 \, M^{2} - 4 & -1 & -2 \, M & M^{2} + 1 & \\
	\end{array}\]
	{where $r_1$ is $uvu = vuv$ and $r_2$ is the longitude definition.}
	By \cref{prop:slcomp}, the space $\mathcal Z(\cM, \Adr)$ has generator
	\[\bma 0,1,\frac 1 {M-M^{-1}},0,6 , \frac 6 {M-M^{-1}},0,0,0\ema\]
	in the 2-dimensional subspace spanned by
	\[\left\{ \bma0,1,\frac 1 {M-M^{-1}},0,0,0,0,0,0\ema, \bma 0,0,0,0,1,\frac 1 {M-M^{-1}},0,0,0 \ema\right\}\] and the slope is
	$s_T(\Adr) = -6.$ In particular it does not depend on $\rho$.
\end{example}

\begin{example}
	\emph{The figure-eight knot}
	\label{ex:8}
	Let $K$ be the figure-eight knot. There is a unique component $X \subset X(\cM)$ containing irreducible characters (see for instance \cite[Examples 1.6.2 and 5.5]{Ben16}). This component is a plane curve given by the equation
	\[\{ 2x^2+y^2-x^2y-y-1 = 0 \} \subset \CC^2,\]
	{where $x$ it the coordinate function given by $\chi \mapsto \chi(m)$. Note that the coordinate function of the longitude is $\chi \mapsto \chi(\ell) = x^4-5x^2+2$.} Using \cite[Th\'eor\`eme 4.1 (ii)]{Porti97} and \cref{prop:slopetors} we compute
	\begin{align*}
	s_K(x,y)^2 =\frac{x^2-4}{(x^4-5x^2+2)^2-4} (4x^3-10x)^2
	= \frac{4 (2x^2-5)^2}{(x^2-5)(x^2-1)}
	\end{align*}
	Expanding the denominator with the relation $x^2 = \frac{y^2-y-1}{y-2}$, we obtain, up to sign
	\[s_K(x,y) = \pm \frac{2 (2x^2-5)(y-2)}{(y-1)(y-3)}.\]
	\end{example}
\section{Slope and A-polynomial}
\label{subsec:slopeApol}

In this section, we express the slope function in terms of the A-polynomial of the knot.
 As mentioned in \cref{subsec:Apol}, $r(X)$ might have $0$-dimensional components but they are omitted in the definition of the A-polynomial.

\begin{theorem}
	\label{th:slopeApoly}
	Let $X\subset X(\cM)$ be an irreducible component such that $r(X)$ has dimension~$1$. For all $\chi \in X$ with $r(\chi)=(L,M)$, the following holds
		\begin{equation*}
		\label{equa:slopeApoly}
		s_K(\chi)=- \frac M L \cdot \frac{\partial_M A(L,M)}{\partial_L A(L,M)},
	\end{equation*}
	where $A(L,M)=A_K(L,M)$ and $\partial_L$ and $\partial_M$ are the partial derivatives.
\end{theorem}

\begin{remark}
	Combining \cref{prop:slopetors} with \cite[Corollaire 4.9]{Porti97}, the result of \cref{th:slopeApoly} follows directly, up to sign, in the case where $X$ has itself dimension $1$. We resolve those two issues. Moreover \cref{th:slopeApoly} does not require the characters in $X$ to be scheme-reduced, and the factors of the A-polynomial might have multiplicities greater than $1$.
\end{remark}

\begin{proof}
	From \cref{prop:Cotangent} it follows that the Lagrangian $\mathcal Z(\cM, \Adr)$ generically identifies with the Zariski conormal bundle of $r(X(\cM))$ in $X(\partial \cM)$.
	Picking local coordinates $\mathfrak l = \log L, \mathfrak m= \log M$ around $r(\chi)$, the kernel of the cotangent map is generated by
	\[d A(e^\mathfrak l, e^\mathfrak m) = \partial_\mathfrak l A(e^\mathfrak l, e^\mathfrak m) d \mathfrak l + \partial_\mathfrak m A(e^\mathfrak l, e^\mathfrak m) d \mathfrak m\]
	in $\CC^2 = \langle d\frak l ,d\frak m  \rangle$.
	Using the chain rule, we obtain that it is generated by the vector
	\[\left(L \frac{\partial A(M,L)}{\partial L}, M \frac{\partial A(M,L)}{\partial M} \right) \]
	and the proposition follows.
\end{proof}

\begin{remark}
	Let $T$ be the right-handed trefoil knot, with $A_T(L,M)=1+LM^6$. \cref{th:slopeApoly} gives
	\[s_{T} = -\frac M L \cdot \frac{6M^5L}{M^6} = -6.\]
	Compare to \cref{ex:trefoil}.
\end{remark}

\section{The slope at an ideal point}
\label{subsec:ideal}
In this section we prove \cref{theo:slope}. The context of this result is the work of Culler--Shalen (see for instance \cite{Sha02}) which associates incompressible surfaces in $\cM$ to ideal points of curves of $X(\cM)$.

Let $X\subset X(\cM)$ be an irreducible component whose image $r(X)=Y$ is a curve in $X(\partial \cM)$, defined as the zero locus of an irreducible factor $P$ of $A_K(L,M)$.
Its function ring is usually denoted by $\CC[Y] = \CC[L,M]/(P)$, and its function field is $\CC(Y) = \operatorname{Frac}(\CC[Y])$.

To any point $y$ in $Y$ one can associate a \emph{discrete valuation} $v$ on the multiplicative group~$\CC(Y)^*$ in the field $\CC(Y)$ of rational functions on $Y$. A discrete valuation $v \colon \CC(Y)^* \to \ZZ$ is a group epimorphism satisfying $v(f+g) \ge \min (v(f), v(g))$.
The valuation associated to a smooth point $y$ is simply the map
\[f \mapsto v_y(f) = \ord_y f\]
given the vanishing order of $f$ at the point $y$.
More generally, the smooth projective model $\overline Y$ of $Y$ is smooth compact curve birational to $Y$, canonically defined up to isomorphism, and the points of $\overline Y$ are bijectively associated to discrete valuations on the function field $\CC(Y)\simeq \CC(\overline Y)$.

An \emph{ideal} point $y$ of $Y$ is a point added "at infinity" in the smooth projective model $\overline Y$, it corresponds to a valuation $v_y$ on $\CC(Y)$ such that not every regular function $f \in \CC[Y]$ has non-negative valuation $v_y(f)$. In other words, some regular functions (at least one) should have poles at $y$.

In \cite{CS83}, Marc Culler and Peter Shalen gave a procedure to construct an \emph{incompressible} surface $\Sigma$ in $\cM$ from the data of an ideal point $x$ in a subcurve $C$ of $X(\cM)$ together with the valuation $v_xÂ \colon \CC(C)^* \to \ZZ$. Not any ideal point $x \in X(\cM)$ yields an ideal point $y=r(x) \in X(\partial \cM)$.

In this special case, the ideal point $y$ in $Y$ gives an incompressible surface in $\cM$ of a particular kind: as observed in \cite[Proposition 3.1]{CCGLS}, the incompressible surface $\Sigma$ must have non-empty boundary $\partial \Sigma \subset \partial \cM$. The curve $\partial \Sigma$ is a finite union of parallel circles in $\partial \cM$ 
and uniquely determines a \emph{boundary slope} in $\QQ \cup \{\infty\}$: the slope of $a \ell + b m$ in $H_1(\partial \cM;\ZZ)$ is the rational number $\frac b a$.

On the other hand, the Newton polygon of $A_K(L,M) = \sum_{i,j} a_{i,j} L^iM^j$ is the convex hull in $\CC^2$ of the points $\{(i,j) \in \ZZ^2 \mid a_{i,j} \neq 0\}$. It is a convex polygon of $\CC^2$ with integral vertices, whose sides have a slope in $\QQ \cup \{\infty\}$.
In \cite{CCGLS}, Culler, Cooper, Gillet, Long and Shalen proved the following result:
\begin{theorem}{\cite[Theorem 3.4]{CCGLS}}
	The slopes of the sides of the Newton polygon of the $A$-polynomial $A_K(L,M)$ are boundary slopes of incompressible surfaces in $\cM$ which correspond to ideal points of one-dimensional components of $r^*(X(\cM))$ in $X(\partial \cM)$.
\end{theorem}

Our next statement (\cref{theo:slope} in the introduction) 
states that the slope invariant studied in this paper coincides with the slopes of \cite{CCGLS} at ideal points.

\begin{theorem}
	\label{theo:slopeproof}
	Let $y$ be an ideal point in a one-dimensional component $Y$ of the $A$-polynomial. Then the value of the slope function at the ideal point $y$ equals \emph{minus} the boundary slope of an incompressible surface corresponding to $y$ or \emph{minus} the slope of the corresponding side of the Newton polygon of the $A$-polynomial.
\end{theorem}

\begin{proof}
	The coordinate functions $L,M$ define rational functions on $Y$, in particular their valuations $v_y(L)$ and $v_y(M)$ are well-defined. Since $y$ is an ideal point and $L,M$ generate the coordinate ring $\CC[Y]$ of the curve $Y$, at least one of this valuation must be negative, 
	and at least one of these coordinate functions must have a pole at $y$.

	\begin{claim}
		The value of $s_K$ at the ideal point $y$ is $\frac{v_y(L)}{v_y(M)}$.
	\end{claim}
	\begin{proof}[Proof of the claim]
		From the proof of \cref{prop:slopetors}, we deduce that the value of the slope at $y$ is given by
		\[s_K(y) = \lim_{(L,M) \to y} \frac {r^*(dL/L)}{r^*(dM/M)}.\]
		The following argument is an algebraic analogue of taking Taylor expansion of the functions $L$ and $M$ around the ideal point $y$.
		We pick $t$ a local coordinate around $y$. It is characterized by $v_y(t) = 1$, and we can write \[L = u_1 t^{v_y(L)}\] for $u_1 \in \CC(Y), v_y(u_1)=0$, and similarly \[M = u_2 t^{v_y(M)}\]  for $u_2 \in \CC(Y), v_y(u_2)=0$.
				Moreover, near $y$ it follows that
		\[\frac {r^*(dL/L)}{r^*(dM/M)} = \frac{v_y(L)/t}{v_y(M)/t} = \frac{v_y(L)}{v_y(M)}\]
		and the claim follows.
	\end{proof}

	Now \cref{theo:slope} follows directly from the claim, because it is proven in \cite[Proposition 3.1]{CCGLS} that the quantity $-\frac{v_y(L)}{v_y(M)}$ is the boundary slope of an incompressible surface corresponding to $y$.
\end{proof}

\bibliography{biblio}
\bibliographystyle{plain}

\end{document}